\documentclass[reqno]{amsart}
\usepackage{epsfig,latexsym,amsfonts,amssymb,amsmath,amscd}
\usepackage{amsthm}
\usepackage[mathscr]{eucal}
\usepackage{mathrsfs}
\usepackage{mathbbol}
\usepackage{oldgerm,units}

\usepackage{epsfig,latexsym,amsfonts,amssymb,amsmath,amscd,graphics,epic}
\usepackage{amsfonts,amssymb,amsmath,amscd,amsthm}
\usepackage[mathscr]{eucal}
\usepackage{mathrsfs}
\usepackage{oldgerm,units}
\usepackage{wrapfig,epsfig}

\usepackage{ifthen}
\usepackage{mathbbol}

\usepackage{amsthm}
\usepackage[mathscr]{eucal}
\usepackage{mathrsfs}
\usepackage{mathbbol}
\usepackage{oldgerm,units}
\usepackage{wrapfig}

\newcommand{\Ref}[1]{(\ref{#1})}

\newcommand{\Real}{\mathbb R}

\newcommand{\Net}{\mathbb N}

\newcommand{\one}{\mathbb{1}}
\newcommand{\zero}{\mathbb{0}}

\newcommand{\trop}[1]{\mathcal{#1}}

\newcommand{\tB}{\trop{B}}

\newcommand{\tG}{\trop{G}}

\newcommand{\tM}{\trop{M}}

\newcommand{\tS}{\trop{S}}
\newcommand{\tT}{\trop{T}}

\newcommand{\al}{\alpha}
\newcommand{\bt}{\beta}

\newcommand{\lm}{\lambda}
\newcommand{\Lm}{\Lambda}

\newcommand{\OP}{\left(}
\newcommand{\CP}{\right)}

\newcommand{\id}{\inva{\aad}}

    \ifx\proof\undefined
    \newenvironment{proof}{
    \smallskip
    \noindent\emph{Proof.}}{\hfill\(\Box\)
    \bigskip
    } \fi

\newcommand{\vMat}[4]{\OP \begin{array}{cc}
  #1 & #2 \\
  #3 & #4
\end{array}\CP}

\newcommand{\ifdef}[3]{\ifthenelse{\equal{#1}{true}}{#2}{#3}}

\newtheorem*{examp*}{Example}

\def\scrR{\mathscr R}
\def\trn{{\operatorname{t}}}
\def\ghost{\text{ghost}}
\def\regular{nonsingular}

\def\lay{\operatorname{lay}}
\def\um{I}
\def\Fz{F}
\def\base{\tB}

 \input amssymb.sty
\def\rank{\operatorname{rank}}
\def\rankS{\operatorname{rank_S}}

\def\bidomain0{bi-\domain0}
\def\bidomains0{bi-\domains0}
\def\bisemifield0{bi-\semifield0}
\def\bisemifields0{bi-\semifields0}

\def\id{\operatorname{id}}

\newcommand{\ds}[1]{\ {#1} \ }

\newcommand{\Inu}[1]{\widehat{#1}}

\input xy
\xyoption{all}

\def\FunSR{\operatorname{Fun} (\tSS,R)}

\def\({\left(}
\def\){\right)}

\def\Rpl{(\mathbb R, + \,)}
\def\Npl{(\Net, + \,)}
\def\Qpl{(\Q, + \,)}
\def\Zpl{(\Z, + \,)}

\def\Real{\mathbb R}
\def\Q{\mathbb Q}

\def\zm{Z}
\def\tdet{determinant}

\def\ltw{0.7\textwidth}

\def\beginA {\pSkip \qquad \begin{minipage}{\ltw}}
\def\endA{\end{minipage} \pSkip}

\newcommand\boxtext[1]{\pSkip \qquad \qquad \qquad \framebox{\parbox{\ltw}{#1}}\pSkip}

\newcommand{\xl}[2]{\,\,{^{[#2]}}{#1}\,}

\newcommand{\Det}[1]{ \left|{#1}\right|}

\newcommand{\etype}[1]{\renewcommand{\labelenumi}{(#1{enumi})}}
\def\eroman{\etype{\roman}}

\def\rat{\operatorname{rat}}

\def\tG{\mathcal{G}}
\def\tSS{\mathcal{S}}
\def\tGz{\mathcal{G}_\zero}
\def\bfi{{\bf i}}

\def\bfa{{\bf a}}

\def\one{\mathbb{1}}
\def\zero{\mathbb{0}}
\def\pSkip{\vskip 1.5mm \noindent}

\def\nucong{\cong_\nu}

\def\nb{\nabla}

\def\pipe{{\underset{{\ \, }}{\mid}}}
\def\pipeGS{{\underset{\operatorname{\, gs }}{\mid}}}
\def\pipeL{{\underset{{L}}{\mid}}}

\def\SpipetWl{{\underset{{S\tilde \ell}}{\mid}}}
\def\SdpipetWl{{\underset{{S \ell;d}}{\mid}}}
\def\SdpipetW2{{\underset{{S \ell;2}}{\mid}}}
\def\pipeWl{{\underset{{\ell}}{\mid}}}
\def\pipeSWl{{\underset{{S\ell}}{\mid}}}
\def\pipeWL{{\underset{L}{\mid}}}
\def\lmodg{\mathrel  \pipeGS \joinrel \joinrel \joinrel =}
\def\lmodL{\mathrel  \pipeL \joinrel \joinrel =}
\def\tilG{\widetilde{G}}

\def\pipe1{{\underset{{1}}{\mid}}}
\def\lmod1{\mathrel  \pipe1  \joinrel \joinrel =}

\def\lmodWl{\mathrel  \pipeWl   \joinrel \joinrel =}
\def\lmodWL{\mathrel  \pipeWL   \joinrel \joinrel =}
\def\SlmodWl{\mathrel  \pipeSWl   \joinrel \joinrel \joinrel =}

\def\SltmodWl{\mathrel  \SpipetWl   \joinrel \joinrel \joinrel =}
\def\SdltmodWl{\mathrel  \SdpipetWl   \joinrel \joinrel \joinrel \joinrel  \joinrel =}
\def\SdltmodW2{\mathrel  \SdpipetW2   \joinrel \joinrel \joinrel \joinrel  \joinrel =}
\def\gperp{ {\perp \joinrel  \joinrel \joinrel  \perp }}

\newcommand{\bil}[2]{\langle{#1},{#2}\rangle}
\def\tGz{\tG_\zero}

\def\tTz{\tT_\zero}

\def\sig{\sigma}
\def\la{\lambda}
\def\al{\alpha}
\def\bt{\beta}

\def\semiring0{semiring$^\dagger$}
\def\semialg0{semialg$^\dagger$}
\def\alt{\operatorname{alt}}
\def\cape{\operatorname{Cap}}
\def\stn{\operatorname{Stn}}
\def\mun{e}
\def\lmodWLnu{\mathrel  \pipeWL   \joinrel \joinrel \equiv_\nu}

\def\tr{\operatorname{tr}}
\def\domain0{domain$^\dagger$}
\def\domains0{domains$^\dagger$}
\def\semirings0{semirings$^\dagger$}
\def\semifield0{semifield$^\dagger$}
\def\semifields0{semifields$^\dagger$}

\def\tT{\mathcal{T}}
\def\tSS{S}

\def\Net{\mathbb N}
\def\Z{\mathbb Z}

\def\Fun{{\operatorname{Fun}}}

\def\one{\mathbb 1}
\def\zero{\mathbb 0}
\def\rone{\one_R}
\def\rzero{\zero_R}

\def\fzero{\zero_F}

\newcommand{\adj}[1]{\,\operatorname{adj}(#1)}

\newtheorem{thm}{Theorem} [section]
\newtheorem{exampl}[thm]{Example}

\newtheorem*{thm*}{Theorem}
\newtheorem{cor}[thm]{Corollary}
\newtheorem{lem}[thm]{Lemma}
\newtheorem{prop}[thm]{Proposition}
\newtheorem*{claim*} {Claim}
\newtheorem{acknowledgment*}[thm] {Acknowledgment}
\newtheorem{examp}[thm]{Example}

    \newtheorem*{remarks*} {Remarks}
 \newtheorem*{remark*}{Remark}
 \newtheorem{defn}[thm]{Definition}
\newtheorem{construction}[thm]{Construction}
\newtheorem{rem}[thm]{Remark}

 \renewcommand{\sectionmark}[1]{}

\renewcommand{\a}{\alpha}

\def\R {R}
 
\begin{document}

\title[Algebraic structures]
{Algebraic structures of tropical mathematics} 
\author[Z. Izhakian]{Zur Izhakian}
\address{Department of Mathematics, Bar-Ilan University, Ramat-Gan 52900,
Israel} \email{zzur@math.biu.ac.il}
\author[M. Knebusch]{Manfred Knebusch}
\address{Department of Mathematics,
NWF-I Mathematik, Universit\"at Regensburg 93040 Regensburg,
Germany} \email{manfred.knebusch@mathematik.uni-regensburg.de}
\author[L. Rowen]{Louis Rowen}
 \address{Department of Mathematics,
 Bar-Ilan University, 52900 Ramat-Gan, Israel}
 \email{rowen@math.biu.ac.il}

\thanks{The research of the first and third authors was supported  by the
Israel Science Foundation (grant No.  448/09).}

\thanks{The research of the first author also was conducted under the auspices of the
Oberwolfach Leibniz Fellows Programme (OWLF), Mathematisches
Forschungsinstitut Oberwolfach, Germany}



\subjclass[2010]{Primary 06F20, 11C08, 12K10, 14T05, 14T99, 16Y60;
Secondary  06F25, 16D25. }

\date{\today}


\keywords{tropical algebra, layered supertropical domains,
polynomial semiring, d-base, s-base, bilinear form}


\begin{abstract}

Tropical mathematics often is defined over an ordered cancellative
monoid $\tM$, usually taken to be $(\mathbb R, +)$ or  $(\mathbb
Q, +)$. Although  a rich theory has arisen from this viewpoint,
cf.~\cite{L1}, idempotent semirings possess a restricted algebraic
structure theory, and also do not reflect certain
valuation-theoretic properties, thereby forcing researchers to
rely often on combinatoric techniques.

 In this paper we describe
an alternative structure, more compatible with valuation theory,
studied by the authors over the past few years, that permits
fuller use of algebraic theory especially in understanding the
underlying tropical geometry. The idempotent max-plus algebra $A$
of an ordered monoid~$\tM$ is replaced by~ $R: = L\times \tM$,
where $L$ is a given indexing semiring (not necessarily with 0).
In this case we say $R$ \textbf{layered} by $L$. When $L$ is
trivial, i.e, $L = \{ 1 \}$, $R$ is the usual bipotent max-plus
algebra. When $L = \{ 1, \infty \}$ we recover the ``standard''
supertropical structure with its ``ghost'' layer. When $L =
\mathbb N  $ we can describe multiple roots of polynomials via a
``layering function'' $s: R \to L$.

Likewise, one can define the layering $s: R^{(n)} \to L^{(n)}$
componentwise; vectors $v_1, \dots, v_m$ are called
\textbf{tropically dependent} if each component of some nontrivial
linear combination $ \sum \a_i v_i$ is a ghost, for ``tangible''
$\a_i \in R$. Then an $n\times n$ matrix has tropically dependent
rows iff its permanent is a ghost.

We explain how supertropical algebras, and more generally layered
algebras,  provide a robust algebraic foundation for tropical
linear algebra, in which many classical tools are available. In
the process, we provide some new results concerning the rank of
d-independent sets (such as the fact that they are semi-additive),
put them in the context of supertropical bilinear forms, and lay
the matrix theory in the framework of identities of
 semirings.
\end{abstract}
\maketitle



\numberwithin{equation}{section}

\section{Introduction}\label{sec:Introduction}
Tropical geometry, a rapidly growing area expounded  for example
in \cite{Gat,IMS,L1,MS,SS}, has been based on two main approaches.
 The most direct passage to tropical mathematics is
via logarithms. But valuation theory has   richer algebraic
applications (for example providing a quick proof of Kapranov's
theorem), and much of tropical geometry is based on valuations on
Puiseux series. The structures listed above are compatible with
valuations, and in \S\ref{Puis} we see how valuations fit in with
this approach.

 In his overview, Litvinov \cite{L2}  describes tropicalization as a process of
dequantization. Thus, one is motivated to develop the algebraic
tools at the tropical level, in order to provide an intrinsic
theory to support tropical geometry and linear algebra. The main
mathematical structure of tropical geometry is the max-plus
algebra, which  is viewed algebraically as an ordered monoid.
Considerable recent activity \cite{CHWW,W} concerns geometry over
monoids, but the ordering provides extra structure which enables
us to draw on classical algebraic structure theory.

The max-plus algebra is fine for answering many combinatoric
questions, but it turns out that a more sophisticated structure is
needed to understand the algebraic structure connected with
valuations. Our overlying objective is to translate ordered
monoids into an algebraic theory
 supporting tropical linear algebra and geometry, using the following approaches:

 \begin{itemize}
 \item Algebraic geometry as espoused by Zariski and
 Grothendieck,
 using varieties and commutative algebra in the context of
 category theory. \pSkip

  \item Linear algebra via tropical dependence, the characteristic polynomial, and
 (generalized)
 eigenspaces. \pSkip

   \item  Algebraic formulations for more sophisticated concepts such
as resultants, discriminants, and Jacobians.
\end{itemize}

This approach leads to the use of polynomials and matrices, which
requires two operations.
 Our task has been to pinpoint the appropriate
category of semirings in which to work, or equivalently, how far
do we dequantize in the process of tropicalization? In this survey
we compare four structures, listed in increasing level of
refinement:

 \begin{itemize}
 \item The max-plus algebra,
\pSkip
 \item Supertropical algebra,
\pSkip\item Layered tropical algebras,  \pSkip  \item Exploded
supertropical algebras.
\end{itemize}

We review the layered algebra in \S\ref{laystr}, compare it to the
max-plus algebra, and then in \S\ref{sec4} survey its linear
algebraic theory, especially in terms of different notions of
bases, proving a   new result (Proposition~\ref{semilattice})
about the   semi-additivity of the rank of d-independent sets  of
a layered vector space. In
 \S\ref{iden} we see how these considerations lead naturally to a theory of
 identities. Due to lack of space, we often refer the reader to
\cite{IzhakianRowen2011Layered,IzhakianKnebuschRowen2011CategoriesI}
for more details.

\section{Algebraic Background}\label{back}

We start by reviewing some notions which may be familiar, but are
needed extensively in our exposition. The basic tropicalization,
or dequantization, involves taking
 logarithms to $\Rpl$, which as explained in \cite{L1}
replaces conventional multiplication by addition, and conventional
addition by the maximum. This is called the \textbf{max-plus
algebra} of $\Rpl$.

\subsection{Ordered groups and monoids}

Recall that a \textbf{monoid} $(\tM,\cdot \; , 1)$ is a set with
an associative operation~$\cdot$ and a unit element $1$. We
usually work with \textbf{Abelian} monoids, in which the operation
is commutative. The passage to the max-plus algebra in tropical
mathematics can be viewed algebraically via ordered groups (such
as $\Rpl$), and, more generally, ordered monoids.

An Abelian monoid $\tM:= (\tM,\cdot \; , 1)$ is
\textbf{cancellative} if $a b = a c$ implies $b=c.$ There is a
well-known localization procedure with respect to a  submonoid $S$
of a cancellative Abelian monoid $\tM$, obtained by taking $\tM
\times S/\sim \, ,$ where $\sim$ is the equivalence relation given
by $(a,s) \sim (a',s')$ iff $as' = a's.$ Localizing with respect
to all of $\tM$ yields its \textbf{group of fractions},
cf.~\cite{B,W}. We say that a monoid $\tM$ is
\textbf{power-cancellative} (called \textbf{torsion-free} by
\cite{W}) if $a^n = b^n $ for some $n\in \Net$ implies $a=b.$ A
monoid $\tM$ is called $\Net$-\textbf{divisible} (also called
\textbf{radicalizible}
 in the tropical literature) if for
each $a\in \tM$ and $m \in \Net$ there is $b\in \tM$ such that $b^m =
a.$ For example, $(\Q, + \,)$ is $\Net$-divisible.

\begin{rem}\label{divclo} The customary way of embedding an Abelian  monoid $\tM$ into
an $\Net$-divisible monoid, is to adjoin $\root m \of a$ for each
$a \in \tM$ and $m \in \Net,$ and define
$$\root m \of a \root n \of b \ds {:=} \root {mn} \of {a^n b^m}.$$
This will be {power-cancellative} if $\tM$ is
{power-cancellative}.
\end{rem}

An \textbf{ordered} Abelian monoid is an Abelian monoid endowed
with a total order satisfying the property:
\begin{equation}\label{ogr} a \le b \quad \text{implies}\quad ga
\le gb ,\end{equation} for all elements $a,b,g$. Any ordered
cancellative Abelian monoid is infinite. 

One advantage of working with ordered monoids and groups is that
their elementary theory is well-known to model theorists. The
theory of ordered $\Net$-divisible   Abelian groups is model
complete, cf.~\cite[p.~116]{M} and~\cite[pp.~35, 36]{Sa}, which
essentially means that every $\Net$-divisible ordered cancellative
Abelian monoid has the same algebraic theory as the max-plus
algebra $(\Q, + \, )$, which is a much simpler structure than
$\Rpl$. From this point of view, the algebraic essence of tropical
mathematics boils down to $\Qpl$. Sometimes we want to study its
ordered submonoid $\Zpl$, or even $\Npl$, although they are not
$\Net$-divisible.

Nevertheless, just as one often wants to study the arithmetic of
$\Q$ by viewing finite homomorphic images of $\Z$, we want the
option of studying finite homomorphic images of the ordered monoid
$\Npl$. Towards this end, we define the $q$-\textbf{truncated
monoid} $\tM=[1,q]:=\{1,2,\dots,q\}$, given with  the obvious
ordering; the sum and product of two elements $k,\ell\in L$ are
taken as usual, if   not exceeding $q-1,$ and is $q$ otherwise. In
other words, $q$ could be considered as the infinite element of
the finite  monoid $\tM$.

\subsection{Semirings without zero}

So far, dequantization has enabled us to pass from algebras to
ordered Abelian monoids, which come equipped with a rich model
theory ready to implement, and as noted above, there is a growing
theory of algebraic geometry over monoids \cite{CHWW}. But to
utilize standard tools such as polynomials and matrices, we need
two operations (addition and multiplication), and   return to the
 language of semirings,  using~\cite{golan92} as a general
 reference. We write $^\dagger$ to indicate that we do not require
 the zero element.

 A
\textbf{\semiring0} $(R,+,\cdot,1)$ is a set~$R$ equipped with
 binary operations $+$ and~$\cdot \;$ such that:
\begin{itemize}
    \item $(R, + )$ is an Abelian semigroup; \pSkip
    \item $(R, \cdot \ , \rone )$ is a monoid with identity element
    $\rone$; \pSkip
    \item Multiplication distributes over addition.
\end{itemize}


A \textbf{\semifield0} is a \semiring0 in which every element is
(multiplicatively) invertible. In particular, the max-plus
algebras $\Zpl$, $\Qpl$, and $\Rpl$ are \semifields0, since $+$
now is the multiplication.

 A \textbf{semiring} is a \semiring0
with a zero element  $\rzero $ satisfying
  $$a + \rzero = a, \quad  a \cdot \rzero
 =\rzero = \rzero \cdot a, \quad  \forall a\in R.$$

 We use \semirings0 instead of semirings since
  the zero element can be adjoined formally, and often is
  irrelevant. For example, the zero element of the max-plus algebra
  would be $-\infty$, which requires special attention.

A \textbf{semifield} is a \semifield0 with a zero element
adjoined. Note that under this definition the customary field $\Q$
with the usual operations is not a semifield,  since $\Q \setminus
\{ 0 \}$ is not closed under addition.

Any ordered Abelian monoid gives rise to a max-plus \semiring0,
where the operations are written~$\odot$ and~$\oplus$ and defined
by:
$$a \oplus b : = \max \{a,b\}; \qquad a \odot b : =  a + b.$$

Associativity and distributivity (of $\odot$ over $\oplus$) hold,
but NOT  negation, since $a \oplus b\neq -\infty$ unless $a = b =
-\infty$.
 Although the circle notation is standard in the
tropical literature, we find it difficult to read when dealing
with algebraic formulae. (Compare $x^4 + 7 x^3 + 4x +1$ with $$x
\odot x \odot x \odot x \oplus 7 \odot x \odot x \odot x \oplus 4
\odot x \oplus 1.)$$ Thus, when appealing to the abstract theory
of semirings we use the usual algebraic notation of $\cdot$ (often
suppressed) and $+$ respectively for multiplication and addition.

The max-plus algebra satisfies the property that $a+b \in \{ a,
b\}$; we call this property \textbf{bipotence}.  In particular,
the max-plus algebra, viewed as a \semiring0, is
\textbf{idempotent} in the sense that $a+a = a$ for all~$a$.
Although idempotence pervades the theory, it turns out that what
is really crucial for many applications is the following fact:

\begin{rem}\label{idpar} In any idempotent \semiring0, if $a+b +c= a$,
then $a+b = a.$ (Proof: $ a = a+b+c = (a+b +c) +b  = a+b$.)
\end{rem}

Let us call such a \semiring0 \textbf{proper}. Note that a proper
semiring cannot have additive inverses other than $\zero,$ since
if $ c+a = \zero$, then $a = a+ \zero = a  + c+ a,$ implying $a =
a+c = \zero.$

Any proper \semiring0 $R$ gives rise to a partial order, given by
$a\le b$ iff $a+c = b$ for some $c\in R$. This is a total order
when the \semiring0 $R$ is bipotent. Thus, the categories of
bipotent \semirings0 and ordered monoids are isomorphic, and each
language has its particular advantages.

 \subsection{The function \semiring0}

 \begin{defn}\label{var}   The \textbf{function \semiring0}   $\FunSR$ is the set of
functions from a set $S$ to a \semiring0~$R$.\end{defn}

 $\FunSR$ becomes a \semiring0 under
componentwise operations, and is proper when $R$ is proper.
Customarily one takes $S = R^{(n)},$ the Cartesian product of $n$
copies of $R$. This definition enables us to work with proper
subsets, but the geometric applications lie outside the scope of
the present paper.


\subsubsection{Polynomials and power series}

%
%

$\Lambda = \{ \la_1, \dots, \la _n \}$ always denotes a finite
 set of indeterminates commuting with the  \semiring0 $R$; often $n=1$ and we have a single
indeterminate $\la$. We have the polynomial \semiring0
$R[\Lambda]$. As in \cite{IzhakianRowen2007SuperTropical}, we view
polynomials in $R[\Lambda]$ as functions, but perhaps viewed over
some extension~$R' $ of~ $R$. More precisely, for any subset $\tSS
\subseteq R^{(n)},$ there is a natural \semiring0 homomorphism
\begin{equation}\label{eq:polymap1} \psi: R[\Lambda] \to \Fun
(\tSS, R),
\end{equation}
 obtained by viewing a polynomial as a function on $\tSS$.

 When $R$ is a \semifield0, the same
 analysis is applicable to Laurent polynomials $R[\Lambda,
 \Lambda^{-1}]$, since the homomorphism $\la_i \mapsto a_i$ then sends $\la_i^{-1} \mapsto a_i^{-1}$.
 Likewise, when $R$ is power-cancellative and divisible, we can also define the \textbf{\semiring0 of rational polynomials}
  $R[\Lambda]_{\rat}$, where the powers of the $\la_i$ are taken to
  be arbitrary rational numbers. These can all be
  viewed as elementary formulas in the appropriate languages, so
  the model theory alluded to earlier is applicable to the appropriate polynomials
  and their (tropical) roots in each case.

Other functions over the bipotent \semiring0 $R$ of an ordered
monoid $\tM$ can be defined in the same way. For example, if $\tM$
 is an ordered submonoid of $(\Real^+, \cdot),$ then we can define the formal exponential series
\begin{equation}\label{exp1} \exp(a) := \sum_k \frac{a^k}{k!}\end{equation}
since $a < m$ implies $\frac{a^{m+1}}{(m+1)!} < \frac{a^m}{m!}, $
and thus \eqref{exp1} becomes a finite sum. It follows at once
that $\exp(\la) :=  \sum \frac{\la^k}{k!}$ is defined in
$\Fun(R,R).$

\subsection{Puisuex series and valuations}\label{Puis}

Since logarithms often do not work well with algebraic structure ,
tropicalists have turned to the \textbf{algebra of Puiseux
series}, denoted $\mathbb K$, whose elements have the form $$ p(t)
= \sum_{\tau \in \mathbb Q_{\ge 0},\ c_\tau \in K} c_{\tau} t
^{\tau},$$ where the powers of $t$ are taken over well-ordered
subsets of $\mathbb Q$. Here $K$ is any algebraically closed field
of characteristic 0, customarily $\mathbb C$. Intuitively, we view
$t$ as a ``generic element.''
 In the literature, the powers~$\tau$ are
 often  taken in $\Real$ rather than $\Q$, but
 it is enough to work with $\Q$, for which
it much easier to compute the powers of $t$. 
The algebra~ $\mathbb
K$  is an algebraically closed field.

Now recall that a \textbf{valuation} from an integral domain $W$
to an ordered monoid $(\tG,+ \, )$
 is a multiplicative monoid homomorphism  $v: W\setminus \{ 0 \}\to \tG
$, i.e., with
$$v(ab) = v(a) + v(b),$$
and   satisfying the   property $v(a+b) \ge \min \{v(a),v(b) \}$
for all $a,b \in K.$ We formally put $v(0) = \infty.$
For example, the field of Puiseux series has the \emph{order
valuation} $v$ given by
\begin{equation*}\label{eq:valPowerSeries}
  v(p(t))\  :=
     \min \{\tau \in  \mathbb Q _{\ge 0} \ : \; c_{\tau} \neq 0 \}.
\end{equation*}
 As $t \to 0,$ the dominant term in $p(t)$ becomes $c_{v({p(t)})} t
^{  v(p(t))}.$

The following basic observation in valuation theory shows why
valuations are  relevant to the tropical theory.

\begin{rem}\label{basicrem} If $v(a) \ne v(b)$, then  $v(a+b) =\min \{v(a),v(b) \}.$
Inductively, if $v(a_1), \dots,  v(a_m)$ are distinct, then

 $$v\bigg(\sum_{i=1}^m a_i\bigg) = \operatorname{min}
\{v(a_i): 1 \le i \le m \}\in \tG.$$ Consequently, if     $\sum
a_i = 0,$ then at least two of the $v(a_i)$ are the same. These
considerations are taken much more deeply in \cite{BiG}.
\end{rem}

When $W$ is a field, the value monoid $\tG$ is a group.
 Much  information about a valuation $v :W \to \tG \cup \{ \infty \}$ can be
garnered from the target $v(W)$, but valuation theory provides some
extra structure:

\begin{itemize}
\item The \textbf{valuation ring}
 $O_v = \{ a \in W: v(a) \ge 0\},$
 \pSkip
 \item The \textbf{valuation ideal}
$P_v = \{ a \in W: v(a) > 0\},$ \pSkip

\item The \textbf{residue ring}
$\bar W = O_v/P_v$, a field if $W$ is a field. \end{itemize}

For example, the valuation ring of the order valuation on the field $\mathbb K$ of
Puiseux series is $\{ p(t)\in  \mathbb K :  c_{\tau} = 0$ for $
\tau < 0\},$ and the residue field is $K$.

 We replace $v$ by  $-v$
to switch minimum to maximum, and $\infty$ by $-\infty$. One can generalize the notion of
valuation to permit $W$ to be a \semiring0;   taking $W = \tM$, we
see that the identity map is a valuation, which  provides one of
our main examples.

\subsection{The standard supertropical \semiring0}

This construction, following
\cite{IzhakianRowen2007SuperTropical},   refines the max-plus
algebra and picks up the essence of the value monoid. From now on, in
the spirit of max-plus, we write the operation of an ordered
monoid~$\tM$ as multiplication.

We start with an Abelian monoid $\tM := (\tM, \cdot \, )$,  an
ordered group $\tG := (\tG, \cdot \,)$, and an onto monoid
homomorphism $v : \tM \to \tG.$
  We write $a^\nu$ for $v(a),$ for $a\in~\tM.$ Thus every element
  of $\tG$ is some $a^\nu$. We write $a \nucong b$
  if $a^\nu = b^\nu.$

Our two main examples:

\begin{itemize}
 \item $\tM = \tG$ is the ordered monoid of the max-plus algebra
 (the
original example in Izhakian's dissertation);
\pSkip
  \item $\tM$ is the
multiplicative group of a field $F$, and $v: F^\times \to \tG$ is
a valuation.  Note that we forget the original addition on the
field~$F!$
\end{itemize}

Our objective is to use the order on $\tG$ to study $\tM$.
Accordingly we want to define a structure on~$\tM \cup \tG$.

The \textbf{standard supertropical \semiring0 } $R$ is the
disjoint union $  \tM \cup \tG$, made into a monoid by starting
with the given multiplications on $\tM$ and $\tG$, and defining $
a \cdot b^\nu $ and $ a^\nu  \cdot b $ to be $ (ab)^\nu$ for $a,b
\in \tM$.
We extend~$v$ to the \textbf{ghost map} $\nu: R\to \tG$ by taking
$\nu|_\tM = v$ and  $\nu|_\tG$ to be the identity on $ \tG$. Thus,
$\nu$ is a monoid projection.

We make $R$ into a \semiring0 by defining $$a+b = \begin{cases} a &
\text{ for } a^\nu > b^\nu;\\b & \text{ for } a^\nu < b^\nu;\\a^\nu &
\text{ for } a^\nu = b^\nu.
\end{cases}$$

 $R$ is never
additively cancellative (except for $\tM = \{\one \}$).

 $\tM$ is called  the \textbf{tangible
submonoid} of $R$.    $ \tG$ is called the \textbf{ghost ideal}.

$R $ is called a \textbf{supertropical \domain0} when the monoid
$\tM$ is (multiplicatively) cancellative.

 Strictly speaking,  a  supertropical \domain0 will not be a \semifield0 since the ghost elements are not
 invertible. Accordingly, we define a \textbf{1-\semifield0} to be a
 supertropical \domain0 for which $\tM$ is a group.

Motivation: The ghost ideal $\tG$ is  to be treated much the same
way that one  treats the zero element in commutative algebra. Towards
this end, we write $$a \lmodg b \qquad \text{ if }\qquad a =b\quad
\text{or}\quad a = b + \text{ghost}.$$ (Accordingly, write $a
\lmodg \zero$ if $a$ is a ghost.) Note that for $a$ tangible, $a
\lmodg b$ iff $a=b.$ If needed, we could formally adjoin a zero
element in a separate component; then the ghost ideal is $\tGz :=
\tG \cup \{\zero \}$. We may  think of the ghost elements as
uncertainties in
 classical algebra
 arising from adding two Puiseux series whose lowest order terms have the same
 degree.

 $R$  is a cover of the  max-plus algebra of $\tG$,
  in which we  ``resolve'' tangible
 idempotence, in the sense that $a+a = a^\nu$ instead of $a+a =
 a.$

 This modification in the structure permits us to detect corner roots of tropical polynomials
 in terms of the algebraic structure, by means of ghosts.  Namely,
 we say that $\bfa \in R^{(n)}$ is a root of a polynomial $f \in
 R[\Lambda]$ when $f(\bfa) \in \tG.$ This concise formulation
 enables us to apply directly many standard mathematical concepts
 from algebra, algebraic geometry, category theory, and model theory, as described in
\cite{IzhakianKnebuschRowen2010Glimpse}--\cite{IzhakianKnebuschRowen2011CategoriesI}
and
\cite{IzhakianRowen2007SuperTropical}--\cite{IzhakianRowen2008Resultants}.

The standard supertropical
 semiring works well with linear algebra, as we shall see.

\subsection{Kapranov's Theorem and the exploded supertropical structure}

Given a polynomial $f(\Lm) = \sum _\bfi p_\bfi ( \la_1^{i_1}
\cdots \la_n^{i_n}) \in \mathbb K[\Lambda],$ where $\bfi = (i_1, \dots, i_n)$,  i.e., with each
$p_\bfi$ a Puiseux series, we define its \textbf{tropicalization}~
$\tilde f $ to be the tropical polynomial $\sum _\bfi v( p_\bfi )
\la_1^{i_1} \cdots \la_n^{i_n}.$ (In the tropical literature, this
is customarily written in the circle notation.) By
Remark~\ref{basicrem}, if $\bfa \in \mathbb K^{(n)}$ is a root of $f$ in the classical
sense, then $v(\bfa)$ is a tropical root of $\tilde f.$ Kapranov
showed, conversely, that any tropical root of $\tilde f$ has the
form $v(\bfa)$ for suitable $\bfa \in \mathbb K^{(n)},$ and
valuation theory can be   applied to give a rather quick proof of
this fact, although we are not aware of an explicit reference.
(See \cite[Proposition 12.58]{R} for an analogous proof of a
related valuation-theoretic result.)

To prove Kapranov's theorem, one needs more than just the lowest
powers of the Puiseux series appearing as coefficients of $f$, but
also their coefficients; i.e., we also must take into account the
residue field of the order valuation on Puiseux series. Thus, we
need to enrich the supertropical structure to include this extra
information. This idea was first utilized by Parker~\cite{Par} in
his ``exploded'' tropical mathematics. Likewise, Kapranov's
Theorem has been extended by Payne~\cite{Payne08,Payne09}, for
which we need the following more refined supertropical structure,
initiated by Sheiner~\cite{Erez2}:

\begin{defn}\label{expl} Given a valuation $v: W \to \tG,$ we define the \textbf{exploded} supertropical
algebra $R =W\times \tG,$ viewed naturally as a monoid. (Thus we
are mixing the ``usual'' world with the tropical world.)

We make $R$ into a \semiring0 by defining $$(c,a)+(d, b)=
\begin{cases} (c,a) & \text{ when } a\ > b;\\(d,b) & \text{ when } a
< b;\\(c+d,a)  & \text{ when } a = b.
\end{cases}$$
\end{defn}

Sheiner's theory parallels the standard supertropical theory,
where now the ghost elements are taken to be the 0-layer $\{ 0 \}
\times \tG.$


\section{The layered structure}\label{laystr}
 The standard
supertropical theory has several drawbacks. First, it
  fails to detect the multiplicity of a
  root  of a polynomial. For example we would want 3 to have multiplicity 5 as a tropical root of
  the tropical polynomial $(\la + 3)^5$; this is not indicated supertropically.
 Furthermore, serious difficulties are encountered when attempting
 to establish a useful intrinsic differential calculus on the supertropical structure.
 Also, some  basic supertropical
 verifications require ad hoc arguments.

These drawbacks are resolved by refining the ghost ideal into
different ``layers,'' following a construction of
\cite[Example~3.4]{WW} and  \cite[Proposition~5.1]{AGG}. Rather
than a single ghost layer, we take an indexing set $L$ which
itself is a partially ordered \semiring0; often $L = \Net$ under
classical addition and multiplication.

 Ordered \semirings0 can be
trickier than ordered groups, since, for example, $a
>b$ in $(\Real,\cdot \; )$ does not imply $-a > -b,$ but rather $-a < -b.$
To circumvent this issue, we require  all elements in the indexing
\semiring0 to be non-negative.

\begin{construction}[{
\cite[Construction~3.2]{IzhakianRowen2011Layered}}]\label{defn5}
Suppose we are given a cancellative
 ordered monoid  $\tG$, viewed as a  \semiring0\ as above. For any partially ordered
 \semiring0~$L$ we define the \semiring0
$R := \scrR(L,\tG)$ to be set-theoretically $L\times \tG$, where
we denote the ``layer'' $\{\ell\} \times \tG$ as $R_\ell$ and the
element $(\ell,a)$ as $\xl{a}{\ell}$; we define multiplication
componentwise, i.e., for $k,\ell\in L,$ $a,b\in\tG,$
\begin{equation}\label{13}   \xl{a}{k} \xl{b}{\ell} =
\xl{(ab)}{k\ell}, \end{equation} and addition via the rules:

\begin{equation}\label{14}
 \xl{a}{k} + \xl{b}{\ell}=\begin{cases}  \xl{a}{k}& \quad\text{if}\ a >
 b,\\ \xl{b}{\ell}& \quad\text{if}\ a <  b,\\
 \xl{a}{k+\ell}& \quad\text{if}\ a= b.\end{cases}\end{equation}
  The \textbf{sort map}
$s: R\to L$ is given by   $s(\xl{a}{k}) = k$.

\end{construction}
 $R $ is indeed a \semiring0. We identify $a\in \tG$ with
$ \xl{a}{1}\in R_1$.

In most applications the ``sorting'' \semiring0 $L$ is ordered,
and its smallest nonzero element  is $1$. In this case, the monoid
$\{ \xl{a}{\ell}: 0 < \ell \le 1 \}$ is called the
\textbf{tangible} part of $R$. The \textbf{ghosts} are
$\{\xl{a}{\ell}: \ell > 1\},$ and correspond to the ghosts in the
standard supertropical theory. The ghosts together with $R_0$
comprise an ideal. If there is a zero element it would be
$\xl{\zero}0$.

 One can view the various choices of the sorting \semiring0 $L$ as
different stages of degeneration of algebraic geometry, where the
crudest (for $L = \{ 1 \}$) is obtained by passing directly to the
familiar max-plus algebra. The   supertropical structure is
obtained when $L=\{1, \infty\}$, where $R_1$ and $R_\infty$ are
two copies of $\tG$, with $R_1$   the tangible submonoid of $R$
and $R_\infty$ being
  the  ghost  copy.
Other useful choices of $L$ include
  $\{1,2,\infty\}$ (to distinguish between simple roots and multiple roots) and
$\Net$, which enables us to work with the multiplicity of roots
and with derivatives, as seen below. In order to deal   with
tropical integration as anti-differentiation, one should consider
the sorting \semirings0 $\Q _{>0}$ and $\Real_{>0}$, but this is
outside our present scope.

By convention,  $\xl{\la }{\ell }$ denotes $\xl{\rone}{\ell}\lm.$
Thus, any monomial can be written in the form $ \xl{\al_\bfi
}{\ell } \lm_1^{i_1}\cdots \lm_n^{i_n}$ where $\bfi =
(i_1,\dots, i_n).$  
We say a polynomial $f$ is \textbf{tangible} if
 each of its coefficients is tangible.

Note that the customary decomposition $R = \bigoplus _{\ell \in L}
R_\ell$ in graded algebras has been strengthened to the partition
$R = \dot\bigcup _{\ell \in L} R_\ell$.
The ghost layers now  indicate the number of
monomials defining a corner root of a tangible
polynomial. Thus,  we can  measure multiplicity of   roots by
means of layers. For example,
$$(\la + 3)^5 = \xl{\la ^5 }{1} +  \xl{3 }{5}\la ^4+  \xl{6
}{10}\la ^3 +  \xl{9}{10}\la ^2 +  \xl{12 }{5}\la  + \xl{15}1,$$
and substituting $3$ for $\la$ gives $\xl{ 15 }{32} = \xl{ 3^5
}{2^5}.$

\subsection{Layered derivatives} Formal derivatives are not very
enlightening over the max-plus
 algebra. For example, if we take the polynomial $f = \la ^2 + 5\la +8$, which has
 corner roots $3$ and $5$, we have $f' = 2 \la + 5,$ having corner root 3, but
 the common corner root 3 of $f$ and $f'$ could hardly be considered a multiple root of $f$.
 This difficulty arises from the fact that $1+1 \ne 2$ in the
 max-plus algebra.
The layering  permits us to define a more useful
version of the derivative (where now $R$ contains a
zero element~$\rzero$):

\begin{defn}
The \textbf{layered derivative} $f'_{\lay}$ of $f$ on $R[\la ]$ is
given by:
\begin{equation}\label{diff2}
 \bigg(\sum_{j=0}^n \xl{\al _j}{\ell_j} \lm ^j\bigg)'_{\lay} := \sum_{j=1}^n  \xl{\al _j}{j \ell_j}  \lm
 ^{j-1} .\end{equation}
\end{defn}

 In particular, for $\a =\xl{\al }{1} \in R_1,$
 $$( \al  \lm^j)'_{\lay} := \xl{\al }{j} \lm^{j-1}
\quad (j\ge 2), \qquad (\al \lm)_{\lay}' := \al , \qquad \text{and
} \ \al '_{\lay} := \rzero.$$

 Thus, we have the familiar formulas:
\begin{enumerate}
 \item  $(f+g)'_{\lay} = f'_{\lay} + g'_{\lay}$; \pSkip

  \item $(fg)'_{\lay} = f'_{\lay} g + fg'_{\lay}$.
\end{enumerate}

This is far more informative in the layered setting (say for $L =
\Net$) than in the standard supertropical setting, in which $(\al
\lm^j)' $ is ghost for all $j\ge 2$.

\subsection{The tropical Laplace transform}

The classical technique of Laplace transforms has a
tropical analog which enables us to compare the various notions of
derivative. Suppose $L$ is infinite, say $L = \Net$ . Formally
permitting infinite vectors $(a_\ell)_{\ell\in L}$ permits us to
define a homomorphism $R[[\Lambda]]\to \scrR(L,R)$ given by
$$ \sum a_k \la ^k \mapsto \big(\xl{k!}k a_k\big).$$ (Strictly speaking,
we would want  the image to be $ (\xl{k!}{\frac{1}{ k }} a_k)
 ,$ but this would complicate the
notation and require us to take $L = \Q^+$.) For example,
$\exp_{\lay}(a)\mapsto (\xl{a_k}k)$ where each $a_k = a.$

Now we define $(\xl{a_\ell}{\ell})' = (\xl{a_\ell}{\ell-1}). $
Then $\exp _{\lay}' =\exp _{\lay}.$ This enables one  to handle
trigonometric functions in the layered theory.

\subsection{Layered \domains0  with symmetry, and
patchworking}\label{basicsym1}

Akian, Gaubert, and Guterman \cite[Definition~4.1]{AGG} introduced
an involutory operation on semirings, which they call a
\textbf{symmetry}, to unify the supertropical theory with
classical ring theory. One can put their symmetry in the context
of $\scrR (L,\tG)$.

\begin{defn}\label{negmap} A \textbf{negation map} on a \semiring0
$L$ is a function $\tau: L \to L$ satisfying the properties:

\boxtext{
\begin{itemize}
\item[N1.] $\tau (k\ell ) = \tau(k)\ell  = k\tau(\ell )$; \pSkip

\item[N2.] $\tau^2 (k) = k$; \pSkip

\item[N3.] $\tau (k+\ell ) = \tau(k) + \tau(\ell )$.
\end{itemize}}
\end{defn}

Suppose the \semiring0  $L$ has a negation  map $\tau$ of order
$\le 2$. We say that   $R: = \scrR (L,\tG)$ has a
 \textbf{symmetry}~$\sig$ when
$R$ is endowed with a map
$$\sig : R\to R$$ and a negation map $\tau$ on $L$, together with the extra axiom:

\boxtext{
\begin{enumerate}\eroman

\item[S1.]   $s( \sig(a)) = \tau(s(a)), \quad \forall a\in \R.$
\end{enumerate} }

\begin{exampl}\label{doubntrunc}  Suppose
$L$ is an ordered \semiring0. We mimic the well-known construction
of $\mathbb Z$ from $\Net.$ Define the \textbf{doubled \semiring0}
 $$D(L) = L_1 \times L_{-1},$$ the direct product of two
copies $L_1$ and $L_{-1}$ of $L$, where addition is defined
componentwise, but multiplication is given by
$$(k,\ell)\cdot (k',\ell') = (kk'+\ell\ell',k \ell'+\ell k').$$ In other words, $D(L)$ is multiplicatively
graded by $\{ \pm 1 \}.$

$D(L)$ is endowed with the product partial order, i.e.,
$(k',\ell') \ge ( k,\ell)$ when  $k'\ge k$ and $\ell' \ge \ell$.
\end{exampl}

Here is an example relating to ``patchworking,'' \cite{IMS}.

\begin{exampl}\label{doubntrunc1} Suppose $\tG$ is an ordered Abelian monoid,   viewed as a \semiring0\ as in
Construction~\ref{defn5}.  Define the \textbf{doubled layered
\domain0}
$$R = \scrR ({D(L)},\tG) = \{ ((k,\ell),a): (k, \ell) \ne
(0,0), \ a \in \tG\},$$  but with addition and multiplication
given by the following rules:

\begin{equation*}\label{141}
\begin{array}{rll}
((k,\ell),a)+ ((k',\ell'),b) & = & \begin{cases} ((k,\ell),a)& \quad\text{if}\ a> b,\\
((k',\ell'),b)& \quad\text{if}\ a< b,\\
((k+k',   \ell + \ell'),\,a)& \quad\text{if}\ a= b.
\end{cases}\\ \\
((k,\ell),a)\cdot  ((k',\ell'),b) & = & ((k k'+ \ell \ell',  k
\ell'+k'\ell),\,ab). \end{array}
\end{equation*}
\end{exampl}
%

\begin{rem} In  $R =
\scrR ({D(L)},\tG),$  the symmetry $\sigma: R \to R$ given by
$\sigma : ((k,\ell),a) \mapsto ( (\ell,k),a)$  is analogous to the
one described in \cite{AGG}, and behaves much like negation.
\end{rem}

For example, when $L = \{ 1, \infty\},$ we note that $D(L) = \{
(1,1), (1,\infty), (\infty,1), (\infty,\infty) \},$ which is
applicable to Viro's theory of patchworking, where the
``tangible'' part could be viewed as those elements of layer
$(1,1), (1,\infty),$ or $(\infty,1)$. Explicitly, comparing with
Viro's use of hyperfields in \cite[\S~3.5]{V}, we  identify these
three layers respectively with $0, 1,$ and $ -1$ in his
terminology, and the element $(\infty,\infty)$ with the set $\{0,
1, -1\}$.

\section{Matrices and linear algebra}\label{sec4}

As an application, the supertropical and layered structures
provide many of the analogs to
 the classical Hamilton-Cayley-Frobenius theory. $M_n(R)$ denotes the \semiring0 of $n
\times n$ matrices over a semiring $R$. (Note that the familiar
   matrix operations do not require negation.)

Although one of the more popular and most applicable aspects of
idempotent mathematics, idempotent matrix theory is handicapped by
the lack of an element $-1$ with  which to construct the
determinant. Many ingenious methods have been devised to
circumvent this difficulty, as surveyed in \cite{ABG};   also
cf.~\cite{AGG} and many interesting papers in this volume.
Unfortunately these give rise to many different notions of rank of
matrix, and often are difficult to understand. The layered (and
more specifically, supertropical) theories give a unified and
relatively straightforward notion of rank of a matrix, eigenvalue,
adjoint, etc.

 \subsection{The supertropical determinant}

 This discussion summarizes
\cite{IzhakianRowen2008Matrices}. We define the
\textbf{supertropical determinant} $|A|$ of a matrix $A =
(a_{i,j})$ to be the permanent:
\begin{equation}\label{det2}
\Det{(a_{i,j})} = \sum _{\pi \in S_n}a_{1,\pi (1)} \cdots a_{n,\pi
(n)}.\end{equation}

 Defining the \textbf{transpose matrix}
$(a_{i,j})^{\trn}$ to be $(a_{j,i}),$ we have
$$\Det{(a_{i,j})^{\trn} } = \Det{(a_{i,j})}.$$

 $\Det{ A } = \rzero$ iff ``enough'' entries are $\rzero$ to force each
summand in Formula \ref{det2} to be $\rzero$.  This property,
which in classical matrix theory provides a description of
singular subspaces, is too strong for our purposes. We now take
the natural supertropical version. Write $\tT$ for the tangible
elements of our supertropical semiring ~$R$, and $\tTz = \tT \cup
\{ \zero \}$.

\begin{defn}
 A matrix $A$ is \textbf{\regular } if $\Det{ A }\in \tT$;
$A$ is~ \textbf{singular} when $\Det{ A } \in \tGz$.
\end{defn}

The standard supertropical structure
 often is sufficient for matrices, since  it enables
 us to distinguish between nonsingular matrices (in which the
tropical $n \times n$ determinant is computed as the unique
maximal product of $n$ elements in one
   track) and singular matrices.

 The tropical determinant is  not
multiplicative, as seen by taking the nonsingular matrix $A = \(
\begin{matrix} 0 & 0 \\ 1 & 2
\end{matrix}\) $. Then $A^2 = \( \begin{matrix} 1 & 2 \\
3 & 4
\end{matrix}\) $ is singular   and $\Det{A^2}
=5^\nu \ne 2\cdot 2$. But  we do have:

\begin{thm}\label{thm:detOfProd} For any $n \times n$ matrices over a supertropical semiring $R$, we have
$$\Det{AB}  \lmodg \Det{A} \Det{B}.$$ In particular, $\Det{AB} = \Det{A}\, \Det{B}$ whenever $\Det{AB} $ is
tangible.
\end{thm}

 We say a
permutation $\sig \in S_n$
 \textbf{attains}~$\Det{A}$ if $\Det{A}\nucong a_{\sig (1),1} \cdots a_{\sig
(n),n}.$

\begin{itemize}
\item By definition, some permutation always attains $\Det{A}$.
\pSkip \item If  there is a unique permutation $\sig$ which
attains $\Det{A}$, then  $\Det{A}= a_{1,\sig(1)} \;  \cdots \;
a_{n,\sig(n)}$.  \pSkip

\item If at least two permutations  attain $\Det{A}$, then $A$
must be singular. Note in this case that if we replaced all
nonzero entries of $A$ by tangible entries of the same
$\nu$-value, then $A$ would still be singular. \pSkip
%
\end{itemize}

\subsection{Quasi-identities and the adjoint}

\begin{defn}  A
\textbf{quasi-identity} matrix $\um_\tG$ is a \regular ,
multiplicatively idempotent matrix equal to $\um + \zm_\tG$, where
$\zm_\tG$ is $\rzero$ on the diagonal, and whose off-diagonal
entries are ghosts or $\rzero$.


\end{defn}

$\Det{
I_\tG } = \rone$ by the \regular ity of $I_\tG$.
Also,
 for any matrix $A$ and any quasi-identity, $I_\tG,$ we have
$AI_\tG = A + A_\tG,$ where $A_\tG = A Z _\tG  \in M_n(\tGz)$.

%
%

%

There is another notion to help us out.
\begin{defn} The $(i,j)$-\textbf{minor} $A'_{i,j}$ of a matrix $A =(a_{i,j})$ is obtained by
deleting the $i$ row and $j$~column of~$A$. The \textbf{adjoint}
matrix $\adj{A}$ of $A$ is defined as the transpose of the matrix
$(a'_{i,j}),$ where $a'_{i,j}= \Det{A'_{i,j}}$.
\end{defn}
%

\begin{rem}\label{rmk:adj}  $ $
\begin{enumerate} \eroman
    \item

Suppose $A = (a_{i,j})$. An easy calculation using Formula
\Ref{det2} yields
\begin{equation}\label{det3}
\Det{A} = \sum_{j=1}^n a_{i,j} \, a'_{i,j}, \quad \forall
i.\end{equation} Consequently, $a_{i,j}  \, {a'}_{i,j} \le  _\nu
\Det{A} $ for each $i,j$.\pSkip

\pSkip \item  If we take $k\ne i,$ then replacing the $i$ row by
the $k$ row in $A$ yields a matrix with two identical rows;  thus,
its \tdet \ is a ghost, and we thereby obtain
\begin{equation}\label{det3.2}
 \sum_{j=1}^n a_{i,j} \, a'_{k,j} \in \tGz, \qquad \forall
k \ne i;\end{equation}
Likewise
\begin{equation*}\label{det3.21}
 \sum_{j=1}^n a_{j,i} \, a'_{j,k} \in \tGz, \qquad \forall
k \ne i.\end{equation*}\pSkip

\end{enumerate}
\end{rem}

 One easily checks that $\adj{B}\adj{A}  =
\adj{AB}$ for any $2\times 2$ matrices $A$ and $B$. However,  this
fails for larger $n$,
cf.~\cite[Example~4.7]{IzhakianRowen2008Matrices}. We do  have the
following fact, which illustrates the subtleties of the
supertropical structure,
cf.~\cite[Proposition~5.6]{IzhakianRowen2008Matrices}:

\begin{prop}  $\adj{AB} =\adj{B}\adj{A}+\ghost.$ \end{prop}

\begin{defn}\label{quas} For $\Det{ A}$  invertible, define $$I_A = A \frac{\adj{A}}{\Det{A}},\qquad  I'_A = \frac{\adj{A}}{\Det{A}}A.$$
\end{defn}

The matrices $\um _ A$ and $\um' _ A$   are quasi-identities, as
seen in \cite[Theorem~4.13]{IzhakianRowen2008Matrices}. The main
technique of proof is to define a \textbf{string}  (from the
matrix $A$) to be a product $a_{i_1, j_1}\cdots a_{i_k, j_k}$ of
entries from $A$ and, given such a string, to define its
\textbf{digraph} to be the graph whose edges are  $(i_1, j_1),
\dots, (i_k, j_k),$ counting multiplicities. A
$k$-\textbf{multicycle}  in a digraph is the union of disjoint
simple cycles, the sum of whose lengths is $k$; thus every vertex
in an $n$-multicycle appears exactly once. A careful examination
of the digraph in conjunction with Hall's Marriage Theorem yields
the following major results from \cite[Theorem~4.9
and~Theorem~4.12]{IzhakianRowen2008Matrices}:

\begin{thm}\label{adjeq} $ $ \eroman
\begin{enumerate}
    \item $\Det{ A \adj{A}}  = \Det{A} ^n.$ \pSkip
    \item $\Det{ \adj{A}}  = \Det{A}^{n-1}.$
\end{enumerate}
\end{thm}

 In case $A$ is a \regular, we define
$$A^\nb  = \frac
{ \adj{A}}{\Det{A}}.$$
Thus $A A^\nb  =\um _ A$, and $A^\nb  A= \um '_ A$. Note that $\um
'_ A$ and $\um _ A$ may differ off the diagonal, although
$$\um _ A A = A  A^\nb  A = A \um' _ A.$$

This result is refined in
\cite[Theorem~2.18]{IzhakianRowen2009MatricesII}. One might hope
 that $A \adj{A} A = \Det{A} A,$  but this is false in
general! The difficulty is that one might not be able to extract
an $n$-multicycle from
\begin{equation}\label{vonNeum} a_{i,j}a'_{k,j}a_{k,\ell}.\end{equation} For example, when $n=3$,  the term
$a_{1,1}(a_{1,3}a_{3,2})a_{2,2}= a_{1,1}a'_{2,1}a_{2,2}$ does not
contain an $n$-multicycle. We do have the following positive
result from \cite[Theorem~4.18]{IzhakianRowen2009MatricesII}:

\begin{thm}\label{doublead}
 $   \adj{A} \adj{\adj{A}}\adj{  A} \nucong |A|^{n-1}\adj{ A}$ for any $n \times n$ matrix $A$.
\end{thm}

\subsection{The supertropical Hamilton-Cayley theorem}

\begin{defn}\label{esschar}
Define the \textbf{characteristic polynomial} $f_A$ of the matrix
$A$ to be
$$ f_A = \Det{\la I+ A},  $$
and the \textbf{tangible characteristic polynomial} to be a
tangible polynomial $\Inu{ f_A} = \la ^n + \sum _{i=1}^{n} \Inu{ \a}
_i \la ^{n-i} $,  where~$\hat \al_i$ are tangible and $\Inu{ \al}_i
\nucong \al_i$, such that $f_A = \la ^n + \sum _{i=1}^{n} \a _i
\la ^{n-i}$.
\end{defn}

 Under this notation, we see that  $\a _k \in R$ arises from the
 dominant
$k$-multicycles in the digraph  of $A$. We say that a matrix $A$
\textbf{satisfies} a polynomial $f\in R[\la]$ if $f(A) \in
M_n(\tGz).$

\begin{thm}\label{hamilton-Cayley}(\textbf{Supertropical Hamilton-Cayley}, \cite[Theorem~5.2]{IzhakianRowen2008Matrices}) Any matrix $A$ satisfies both its characteristic
polynomial~ $f_A$ and its tangible characteristic polynomial $\Inu{
f_A}$.\end{thm}
%

\subsection{Tropical dependence}

Now we apply supertropical matrix theory to vectors.
 As in
classical mathematics, one  defines a \textbf{module} (often
called \textbf{semi-module} in the literature) analogously to
module in classical algebra, noting again that negation does not
appear in the definition.  It is convenient to stipulate that the
module~$V$ has a zero element $\zero_V$, and then we need the
axiom:
$$ a \zero_V = \zero_V \ds{\text{ for all }} a \in R.$$

\noindent Also, if $\zero \in R$ then we require that $\zero v =
\zero_V$ for all $v\in V$.

In what follows, $F$ always denotes  a
 1-semifield. In this case, a module over $F$ is called a (supertropical) \textbf{vector
space}. The natural example is $F^{(n)},$ with componentwise
operations. As in the classical theory, there is the usual
familiar correspondence between the  semiring $M_n(F)$  and the
linear transformations of~$F^{(n)}$.

 For $v = (v_1, \dots, v_n), w = (w_1, \dots,
w_n) \in F^{(n)}$, we write $v \lmodg w$ when $v_i \lmodg w_i $
for all $1 \le i \le n.$

 Here is an application of the adjoint
matrix, used to solve equations.

\begin{rem} Suppose $A$ is  \regular, and $v \in F^{(n)}.$
Then the equation $Aw = v + \text{ghost}$ has the solution $w =
A^\nb v.$ Indeed, writing $I_A = I + Z_\tG$ for a ghost matrix
$Z_\tG$, we have $$Aw = AA^\nb v = I_A v = (I + Z_\tG)v \lmodg
v.$$\end{rem} This leads to
 the supertropical
analog of Cramer's rule
\cite[Theorem~3.5]{IzhakianRowen2009MatricesII}:

\begin{thm}\label{tangsol} If $A$ is a \regular\  matrix and $v$ is a tangible
vector, then the equation $A x \lmodg  v$ has a solution over $F$
which is the tangible vector having value $A^\nb v .$
\end{thm}

 Our next task is to characterize singularity of a matrix $A$ in terms
of ``tropical dependence'' of its rows. In some ways the standard
supertropical theory works well with matrices, since we are
interested mainly
 in whether or not this matrix is nonsingular, i.e., if its
 determinant is tangible; at the outset,  at least, we are
 not concerned with the precise ghost layer of the determinant.

\begin{defn}\label{tropdep1}  A subset $W \subset
F^{(n)}$ is \textbf{tropically dependent} if there is a finite sum
$\sum \a_i w_i \in \tGz^{(n)}$, with each $\a_i \in \tTz$, but not
all of them $\rzero$; otherwise $W\subset F^{(n)}$ is called
\textbf{tropically independent}. A vector $v\in F^{(n)}$ is
\textbf{tropically dependent} on $W$ if $W \cup \{ v \}$ is
tropically dependent.
\end{defn}

By \cite[Proposition~4.5]{IKR2}, we have:

\begin{prop}\label{getbase}  Any $n+1$ vectors in $F^{(n)}$ are tropically
dependent. \end{prop}

\begin{thm}\label{thm:base}(\cite[Theorem
6.5]{IzhakianRowen2008Matrices})   Vectors $v_1, \dots, v_n \in
F^{(n)}$ are tropically dependent, iff   the matrix whose rows are
$v_1, \dots, v_n$ is singular.
\end{thm}

\begin{cor}\label{coldep}   The matrix $A\in M_n(F)$
over a supertropical domain $F$  is \regular \ iff the rows of $A$
are tropically independent, iff the columns of $A$ are tropically
independent.
\end{cor}
\begin{proof} Apply the theorem to $\Det{A}$ and  $\Det{A^t}$,
which are the same.
\end{proof}

There are two competing supertropical notions of base of a vector
space, that of a maximal independent set of vectors, and that of a
minimal spanning set, but this is unavoidable since, unlike the
classical theory, these two definitions need not coincide.

\subsection{Tropical bases and rank}\label{dbasea}

The customary definition of tropical base, which we call
\textbf{s-base} (for \textbf{spanning base}),  is a minimal
spanning set (when it exists). However, this definition is rather
restrictive, and a competing notion  provides a richer theory.

\begin{defn}  A \textbf{d-base} (for \textbf{dependence base}) of a vector space
$V$ is a maximal  set of tropically independent elements of $V$.
A \textbf{d,s-base} is a  d-base which is also an  s-base. The
\textbf{rank} of a set $\tB\subseteq V$, denoted $\rank(\tB)$, is
the maximal number of $d$-independent vectors of $\tB.$
 \end{defn}

Our d-base corresponds to the ``basis'' in
\cite[Definition~5.2.4]{MS}. In view of Proposition~\ref{getbase},
all d-bases of $\Fz ^{(n)}$ have precisely $n$ elements.

This leads us to the following definition.

\begin{defn} The \textbf{rank}  of a vector space
$V$ is defined as:
$$ \rank(V):= \max \big \{ \rank(\base) : \base \text{ is a d-base of } V \big \}.$$
\end{defn}

We have just seen that $\rank(\Fz  ^{(n)}) = n.$ Thus, if
$V\subset \Fz ^{(n)}$, then $\rank(V) \le n$.

We might have liked $\rank(V)$ to be independent of the choice of
d-base of $V$,
 for any vector space $V$. This is proved in the
classical theory of vector spaces by showing that dependence is
transitive. However, transitivity of dependence fails in the
supertropical theory, and, in fact, different d-bases may contain
different numbers of elements, even when tangible. An example is
given in \cite[Example~5.4.20]{MS}, and reproduced in
\cite[Example~4.9]{IKR2} as being a subspace of $F^{(4)}$ having
d-bases both of ranks 2 and 3.
%

\begin{examp}
The matrix $A = \left( \begin{matrix} 4 & 4 & 0 \\ 4 & 4 & 1 \\
4 & 4& 2
\end{matrix}\right) $ has rank 2, but   is ``ghost annihilated''   by the tropically independent
vectors $v_1 = (1, 1, 0)^\trn$ and $v_2 = (1, 1, 1)^\trn$; i.e.,
$Av_1 = Av_2 = (5^\nu,5^\nu,5^\nu)^\trn$, although $2+2>3$.
\end{examp}

We do have some consolations.

\begin{prop}[{\cite[Proposition~4.11]{IKR2}}]\label{thm:tangibleRank}  For any tropical subspace $V$ of
$\Fz  ^{(n)}$ and any tangible $v\in V,$ there is a tangible
d-base of $V$ containing $v$ whose rank is that of $V$.
\end{prop}

\begin{prop}[{\cite[Proposition~4.13]{IKR2}}]\label{sumdep1}
Any $n\times n$ matrix of rank $m$ has ghost annihilator of rank
$\ge n-m$.
\end{prop}%

\subsubsection{Semi-additivity of rank}

\begin{defn} A  function  $\rankS : S \to \Net$ is \textbf{monotone}
if for all $S_2 \subseteq S_1   \subseteq S$ we have
  \begin{equation}\label{rankless}  \rankS ( S_2 \cup \{ s \}) - \rankS (S_2) \ge
\rankS (S_1 \cup \{ s \}) - \rankS (S_1)\end{equation} for all $s
\in S$.\end{defn} Note that \eqref{rankless} says that $\rankS
(S_1) - \rankS (S_2) \ge \rankS (S_1 \cup \{ s \})- \rankS ( S_2
\cup \{ s \}).$ Also, taking
$S_2 = \emptyset$    yields $ \rankS (S_1 \cup \{ s \}) - \rankS
(S_1) \le 1$.
\begin{lem} If $\rankS : S \to \Net$ is monotone, then   \begin{equation}\label{addone} \rankS(S_1) +
\rankS(S_2) \ge \rankS(S_1 \cup S_2) + \rankS (S_1 \cap
S_2)\end{equation} for all $S_1, S_2 \subset S.$ \end{lem}
\begin{proof} Induction on $m = \rankS (S_2 \setminus S_1)$. If $m = 0,$ i.e., $S_2
\subseteq S_1$, then the left side of \eqref{addone} equals the
right side. Thus we may assume that $m \ge 1.$ Pick $s $ in a
d-base of  $S_2 \setminus S_1$. Let $S_2' = S_2 \setminus \{ s
\}$. Noting that $\rankS (S_2' \setminus S_1)= m-1,$ we see by
induction that   \begin{equation}\label{addone1} \rankS(S_1) +
\rankS(S_2') \ge \rankS(S_1 \cup S_2') + \rankS (S_1 \cap S_2'),
\end{equation}
or (taking $ S_1 \cup S_2'$ instead of $S_2$ in \eqref{rankless}),
$$\rankS(S_1) - \rankS (S_1 \cap S_2) = \rankS(S_1) - \rankS (S_1
\cap S_2') \ge \rankS(S_1 \cup S_2') - \rankS(S_2') $$ $$\qquad
\qquad \qquad \qquad \qquad  \ge
\rankS(S_1 \cup S_2) - \rankS(S_2) ,$$ 
 yielding
\eqref{addone}.
\end{proof}

\begin{prop}\label{semilattice}  $\rank(S_1) +
\rank(S_2) \ge \rank(S_1 \cup S_2) + \rank (S_1 \cap S_2) $ for
all $S_1, S_2 \subset S.$\end{prop}
\begin{proof}
   $\rank$ is a monotone function, since each side of \eqref{rankless} is 0
or 1, depending on whether or not $s$ is independent of $S_i$, and
only decreases as we enlarge the set.
\end{proof}

%
%
%

\subsection {Supertropical  eigenvectors}

The standard definition of an \textbf{eigenvector} of a matrix $A$
is a vector $v$, with \textbf{eigenvalue} $\bt$, satisfying $Av =
\bt v$. It is well known \cite{brualdi} that any (tangible) matrix
has an eigenvector.

\begin{examp}\label{eigen}  The characteristic polynomial $f_A$
of $$A =\vMat{4}{0}{0}{1}$$
 is $(\la+4)(\la+1) +
0 = (\la+4)(\la+1),$ and  the vector $(4,0)$ is a
 eigenvector of $A$, with eigenvalue 4. However, there is no
eigenvector having eigenvalue 1.
\end{examp}

In general, the lesser roots of the characteristic polynomial are
``lost" as eigenvalues. We rectify this deficiency by weakening
the standard definition.

\begin{defn}\label{eigen3}
A tangible vector $v$ is a \textbf{generalized supertropical
eigenvector} of a (not necessarily tangible) matrix $A$, with
\textbf{generalized supertropical eigenvalue} $\bt \in \tTz$,
 if $A^m v \lmodg \bt^m v  $ for some $m$;
the minimal such $m$ is called the \textbf{multiplicity} of the
eigenvalue (and also of the eigenvector). A \textbf{supertropical
eigenvector} is a
 generalized   supertropical eigenvector  of multiplicity 1.
\end{defn}

\begin{examp}\label{eigen4} The matrix
$A =\vMat{4}{0}{0}{1}$ of ~Example \ref{eigen} also has the
tangible supertropical eigenvector $v = (0, 4)$, corresponding to
the supertropical
 eigenvalue~$1$, since $$Av = (4^\nu, 5) = 1   v +(4^\nu, -\infty).$$
\end{examp}

%
%

\begin{prop}\label{eigen5} If $v$ is a tangible  supertropical eigenvector of
$A$ with  supertropical eigenvalue $\bt$, the matrix $A + \bt I$
is singular (and thus $\bt$ must be a (tropical) root of the
characteristic polynomial $f_A$ of $A$).
\end{prop}

Conversely, we have:

\begin{thm}[{\cite[Theorem~7.10]{IzhakianRowen2008Matrices}}]\label{eigen8} Assume that $\nu|_\tT : \tT \to \tG$ is 1:1. For any  matrix $A$,
the dominant tangible root of the  characteristic polynomial of
$A$ is an eigenvalue of $A$, and has  a tangible eigenvector. The
other tangible roots   are precisely the supertropical eigenvalues
of $A$.
\end{thm}

Let us return to our example
  $A =  \left(\begin{matrix} 0 & 0 \\
1 & 2
\end{matrix}\right).$
 Its characteristic polynomial is $\la^2 +
2\la +2 = (\la+0)(\la+2),$ whose roots are $2$ and $0$. The
eigenvalue $2$ has tangible eigenvector $v = (0, 2)$  since $Av =
(2, 4) = 2v$, but there are no other tangible eigenvalues. $A$
does have the tangible supertropical eigenvalue $0$, with tangible
supertropical eigenvector $w = (2,1),$ since $Aw = (2, 3^\nu) = 0w
+ ( -\infty, 3^\nu).$
Note that $A + 0I = \left(\begin{matrix} 0^\nu & 0 \\
1 & 2
\end{matrix}\right)$ is singular, because $|A + 0I| = 2^\nu $.

Furthermore, $A^2 =  \left(\begin{matrix} 1 & 2 \\
3 & 4
\end{matrix}\right)$ is a root of $\la^2 + 4A,$ and thus $A $ is a
root of $g = \la^4 + 4\la^2 = (\la(\la+2))^2,$ but~$0$ is not a
root of $g$ although it is a root of $f_A$. This shows that the
naive formulation of Frobenius' theorem fails in the supertropical
theory, and is explained in the work of Adi Niv~\cite{N}.

\subsection{Bilinear forms and orthogonality}

One can refine the study of bases by introducing angles, i.e.,
orthogonality, in terms of bilinear forms. Let us quote some
results from \cite{IKR2}.

\begin{defn} \label{12} A \textbf{(supertropical) bilinear form} $B$ on a (supertropical) vector space  $V  $
 is a function $B : V\times V \to \Fz$ satisfying
$$ B(v_1 + v_2, w_1 +  w_2) \lmodg B(v_1,w_1) + B(v_1,w_2) +
B(v_2,w_1) + B(v_2,w_2) ,$$
$$
B(\a v_1, w_1) = \a  B(v_1 ,w_1 )=   B(v_1 ,\a w_1), $$ for all
$\a \in \Fz$ and $v_i\in V,$ and $w_j\in V'.$
\end{defn}

We work with a fixed bilinear form $B = \bil {\phantom w}{
\phantom v}$ on a (supertropical) vector space $V \subseteq
\Fz^{(n)} $.
 The \textbf{Gram matrix} of vectors $v_1, \dots,
v_k \in \Fz^{(n)}$ is defined as the $k \times k$  matrix
\begin{equation}\label{eq:GramMatrix}
\tilG(v_1, \dots, v_k ) = \left( \begin{array}{cccc}
                      \bil {v_1}{v_1} &  \bil {v_1}{v_2} & \cdots & \bil {v_1}{v_k} \\[1mm]
                      \bil {v_2}{v_1} &  \bil {v_2}{v_2} & \cdots & \bil {v_2}{v_k} \\[1mm]
                         \vdots & \vdots &  \ddots & \vdots \\[1mm]
                      \bil {v_k}{v_1} &  \bil {v_k}{v_2} & \cdots & \bil {v_k}{v_k} \\
                        \end{array} \right).
\end{equation}
The set $\{v_1, \dots, v_k\}$ is \textbf{nonsingular} (with
respect to $B$) when its Gram matrix is nonsingular.

In particular, given a vector space $V$ with s-base $\{ b_1,
\dots, b_k\}$, we have the matrix $\widetilde G = \tilG(b_1,
\dots, b_k ) $, which can be written as $(g_{i,j})$ where $g
_{i,j} = \bil {b_i}{b_j}.$ The  singularity of $\tilG$ does not
depend on the choice of s-base.

\begin{defn} For vectors $v,w$ in $V$, we write
 $v \gperp w$
when $\bil{v}{w} \in \tGz$,  that is $\langle v,w\rangle\lmodg
\fzero$, and say that~$v$ is \textbf{left ghost orthogonal} to
$w$.  We write
  $W^ \gperp $ for $ \{v \in V:
 v \gperp w$ for all $w \in W.\}$ 
\end{defn}

\begin{defn} A subspace $W$ of $V$ is called \textbf{nondegenerate} (with
respect to $B$), if  $W^\gperp \cap W $ is ghost. The bilinear
form $B$ is \textbf{nondegenerate} if the space $V$ is
 nondegenerate.
\end{defn}

\begin{lem}\label{Gramrek} Suppose $\{w_1, \dots, w_m \}$ tropically spans
a subspace $W$ of $V$, and $v\in V.$ If $ \sum_{i=1}^m \bt_i \bil
v{w_i} \in \tGz$ for all $\bt_i \in \tT$, then
 $v\in W^\gperp.$
\end{lem}

\begin{thm}\label{thm:GramMat}(\cite[Theorem~6.7]{IKR2}) Assume that vectors
 $w_1, \dots, w_k \in V$  span  a nondegenerate subspace~$W$ of $V$.
If $ | \tilG(w_1, \dots, w_k ) | \in \tGz,$ then $w_1, \dots, w_k
$ are tropically dependent.
\end{thm}

\begin{cor}\label{nondeg1}
If the bilinear form $B$ is nondegenerate on a vector space $V$,
then   the Gram matrix   (with respect to any given supertropical
d,s-base of $V$) is nonsingular.
\end{cor}

\begin{defn}
The bilinear form $B$ is \textbf{supertropically alternate} if
$\bil vv \in \tGz$ for all $v\in V.$
 $B$ is \textbf{supertropically symmetric}  if
$\bil vw + \bil wv\in \tGz$ for all $v,w \in V$.
\end{defn}

 We aim for the supertropical version (\cite[Theorem~6.19]{IKR2})
of a
 classical theorem of Artin, that any bilinear form
in which ghost-orthogonality is symmetric must be  a
supertropically symmetric bilinear form.

 \begin{defn}\label{os} The (supertropical) bilinear form $B$ is \textbf{orthogonal-symmetric} if it
 satisfies the following property   for any finite sum, with $v_i,w
\in V$:

\begin{equation}\label{orthsym} \qquad \sum _i \bil{v_i}{w } \in \tGz \quad  \text{iff}\quad
\sum _i \bil{w }{v_i} \in \tGz, \end{equation}

 $B$ is \textbf{supertropically orthogonal-symmetric}  if $B$ is
orthogonal-symmetric and satisfies the additional property that
$\bil vw \nucong \bil wv $ for all $v,w \in V$ satisfying $\bil vw
\in \tT.$
\end{defn}
%
%

The symmetry condition extends to sums, and after some easy lemmas
we obtain (\cite[Theorem~6.19]{IKR2}):

\begin{thm}\label{orthogsym} Every orthogonal-symmetric bilinear form~$B$ on a vector space
$V$  is  supertropically symmetric.
\end{thm}

\section{Identities of semirings, especially matrices}\label{iden}
The word ``identity'' has several interpretations, according to
its context. First of all, there are well-known matrix identities
such as the Hamilton-Cayley identity which says that any matrix is
a root of its characteristic polynomial.

Since the
classical theory of polynomial identities is tied in with
invariant theory, we also introduce layered polynomial identities
(PIs), to enrich our knowledge of layered matrices.

\subsection{Polynomial identities of \semirings0}\label{PI}

We draw on basic concepts of polynomial identities, i.e., PI's,
say from \cite[Chapter 23]{Row}. Since \semirings0 do not involve
negatives, we modify the definition a bit.

\begin{defn} The  \textbf{free $\Net$-\semiring0} $\Net\{ x_1, x_2, \dots \}$ is the monoid \semiring0
of the free (word) monoid  $\{ x_1, x_2, \dots \}$ over the
commutative \semiring0 $\Net$.
\end{defn}

\begin{defn}\label{PI2} A \textbf{(\semiring0) polynomial identity} (PI) of a \semiring0 $R$    is a pair $(f,g)$ of (noncommutative) polynomials $f(x_1,\dots,x_m), g(x_1,\dots,x_m) \in
\Net \{x_1,\dots , x_m\}$  for which $$f(r_1,\dots,r_m)=
g(r_1,\dots,r_m),\quad \forall  r_1,\dots, r_m  \in R .$$
 We write
$(f,g) \in \id(R)$ when $(f,g)$ is a PI of $R$.
\end{defn}
%

\begin{rem}\label{PI20} A \textbf{semigroup   identity}  of a semigroup $ \tS$
is
  a pair $(f,g)$ of (noncommutative) monomials $f(x_1,\dots,x_m), g(x_1,\dots,x_m) \in
\Net \{x_1,\dots , x_m\}$  for which $f(s_1,\dots,s_m)=
g(s_1,\dots,s_m),$\ $\forall  s_1,\dots, s_m  \in \tS $. If~$\tS$
is contained in the multiplicative semigroup of a \semiring0 $R$,
the semigroup   identities of  $\tS$ are precisely the \semiring0
PIs $(f,g)$ where $f$ and $g$ are monomials.
\end{rem}

Akian, Gaubert and Guterman \cite[Theorem~4.21]{AGG} proved their
\textbf{strong transfer principle}, which immediately implies the
following easy but important observation:

\begin{thm}\label{meta2} If $f, g \in \Net \{x_1, \dots, x_n\}$ have disjoint supports and $f - g$ is
a PI of $M_n(\Z),$ then $f=g$ is also a \semiring0 PI of
$M_n(R)$ for any commutative \semiring0 $R$.\end{thm}
\begin{proof} Since $\Z$ is an infinite integral domain, $f - g$ is also a
PI of  $M_n(C),$ where $C = \Z[\xi_1, \xi_2, \dots ] $ denotes the
free commutative ring in countably many indeterminates, implying
$(f,g)$ is a \semiring0 PI of $M_n( \Net[\xi_1, \xi_2, \dots ] )$.
But the \semiring0 $M_n(R)$ is a homomorphic image of   $M_n(
\Net[\xi_1, \xi_2, \dots ] )$, implying $(f,g) \in \id(M_n(
R ))$.
 \end{proof}

\begin{cor}\label{meta3} Any PI of $M_n(\Z)$ yields a corresponding
 \semiring0 PI of
$M_n(R)$ for all commutative \semirings0 $R$. \end{cor}
\begin{proof} Take $f$ to be the sum of the terms having positive coefficient, and $g$ to be the sum of the terms having negative coefficient,
and apply the theorem.
 \end{proof}

 Many (but not all) matrix PIs can be viewed  in terms of
 Theorem~\ref{meta2}, although semiring versions  of basic results
such   as   the Amitsur-Levitzki Theorem and Newton's Formulas
often are more transparent here.

 We say that
polynomials $f(x_1,\dots, x_m)$ and $g(x_1,\dots, x_m)$ are a
$t$-\textbf{alternating pair} if $f$ and $g$ are interchanged
  whenever we interchange a pair $x_i$ and $x_j $ for some $1
\le i < j \le t.$ For example, $x_1x_2$ and $x_2x_1$ are a
2-alternating pair. Sometimes we write the
 non-alternating variables as $y_1,y_2,\dots$;  we  write $y$ as
shorthand for all the $y_j.$

\begin{defn}
We partition the symmetric group $S_t$ of permutations in $t$
letters into the even permutations~$S_t^+ $ and  the odd
permutations $S_t^-.$
 Given a $t$-linear polynomial $h(x_1, \dots, x_t;
y)$, we define the \textbf{$t$-alternating pair}
$$h_{\alt}^+(x_1, \dots, x_t; y) := \sum _{\sig \in S_t^+}
h (x_{\sig (1)}, \dots, x_{\sig (t)}; y)$$ and $$h_{\alt}^-(x_1,
\dots, x_t; y) := \sum _{\sig \in S_t^-} h (x_{\sig (1)}, \dots,
x_{\sig (t)}; y).$$

 The \textbf{standard pair} is $\stn_t := (h_{\alt}^+,h_{\alt}^-)$,
where $h=x_1 \cdots x_t.$  Explicitly,
$$\stn_t :=
\bigg(\sum_{\sig \in S _t^+} \, x_{\sig (1)} \cdots  x_{\sig
(t)},\sum_{\sig \in S _t^-} \, x_{\sig (1)} \cdots  x_{\sig
(t)}\bigg).$$ The \textbf{Capelli pair} is $\cape_t :=
(h_{\alt}^+,h_{\alt}^-)$, where $h=x_1 y_1x_2y_2\cdots x_ty_t.$
Explicitly,
$$\cape_t := \bigg(\sum_{\sig \in S _ t^+} \, x_{\sig(1)}y_1x_{\sig (2)}y_2\cdots y_{t-1} x_{\sig (t)}y_{t},\sum_{\sig \in S _ t^-} \, x_{\sig(1)}y_1x_{\sig (2)}y_2\cdots y_{t-1} x_{\sig (t)}y_{t}\bigg).$$
\end{defn}

\begin{prop} Any $t$-alternating pair $(f,g)$ is a
PI for every \semiring0 $R$ spanned by fewer than~$t$ elements
over its center.\end{prop}
\begin{proof} Suppose $R$ is spanned by $\{ b_1, b_2, \dots, b_{t-1}\}.$
We need to verify $$f\bigg(\sum \a_{i,1}b_{i_1} , \dots,  \sum
\a_{i,t}b_{i_t} , \dots \bigg) = g\bigg(\sum \a_{i,1}b_{i_1} ,
\dots, \sum \a_{i,t}b_{i_t} , \dots \bigg) .$$ Since $f$ and $g$
are linear in these entries, it suffices to verify
\begin{equation}\label{ver1}  f( b_{i_1} , \dots,  b_{i_t} , \dots) = g(b_{i_1} , \dots,  b_{i_t} , \dots)  \end{equation}
for all $i_1, \dots, i_t.$ But by hypothesis, two of these must be equal, say $i_k$ and $i_{k'}$,
so switching these two yields~\eqref{ver1} by the alternating hypothesis.
\end{proof}

 Let
 $\mun_{i,j}$ denote the matrix units. The \semiring0 version of the Amitsur-Levitzki theorem
\cite{AL}, that $\stn_{2n}\in \id(M_n(\Net))$, is  an immediate
consequence of Theorem \ref{meta2}, and its minimality follows
from:

 \begin{lem} Any pair of
multilinear polynomials $f(x_1, \dots, x_m)$ and $g(x_1, \dots,
x_m)$ having no common monomials do not comprise a PI of $M_n(R)$
unless $m \ge 2n $.\end{lem}  \begin{proof} Rewriting indices we
may assume that $x_1 \cdots x_m$ appears as a monomial of $f,$ but
not of $g,$ and we note  (for $\ell = \left[\frac m 2\right] + 1$)
that
$$
\begin{array}{lll}
& f(\mun_{1,1}, \mun_{1,2}, \mun_{2,2}, \mun_{2,3}, \dots,
\mun_{k-1,k}, \mun_{k,k}, \dots)  & = \mun_{1,\ell}   \ne 0, \\[1mm]
  \text{ but} \quad & g(\mun_{1,1}, \mun_{1,2}, \mun_{2,2},
\mun_{2,3}, \dots, \mun_{k-1,k}, \mun_{k,k}, \dots) & = 0.
\end{array}
 $$
 \end{proof}

 Likewise, the identical proof of \cite[Remark 23.14]{Row} shows that the Capelli pair $\cape_{n^2}$ is not a PI of~$M_n(C),$ and in fact $(\mun_{1,1},0) \in \cape_{n^2}(M_n(R))$ for any
\semiring0 $R.$

\subsection{Surpassing identities}\label{layPI}

 The \textbf{surpassing identity}  $f\lmodg g$
holds  when $f(a_1, \dots, a_m) \lmodg g(a_1, \dots, a_m)$ for all
 $a_1, \dots, a_m \in R$.

\begin{examp} Take the general $2 \times 2$ matrix $A = \(\begin{matrix} a & b \\
c & d
\end{matrix}\).$ Then $\tr(A)= a+d$ and $\Det{A} = ad+bc.$
$A^2 = \(\begin{matrix} a^2 + bc & b(a+d) \\
c(a+d) & bc +d^2
\end{matrix}\),$ so
$$A^2 + adI =  \(\begin{matrix} a(a+d) + bc & b(a+d) \\
c(a+d) & bc +d(a+d)
\end{matrix}\)  =  \tr(A)A +  bcI  ,$$ implying
 $$A^2 + \Det{A}I =  \tr(A)A +  bc^\nu I  ,$$
yielding the   surpassing identity  $A^2 + \Det{A}I \lmodg
\tr(A)A$ for $2 \times 2$ matrices.
\end{examp}

We might hope for a surpassing identity involving alternating
terms in the Hamilton-Cayley polynomial, but a cursory examination
of matrix cycles dashes our hopes.

\begin{examp} Let $A = \(\begin{matrix} - & d & a \\
c & - & - \\
- & b  & -
\end{matrix}\).$
Then $ A^2 =  \(\begin{matrix} cd & ab   & - \\
- & cd & ac \\
 bc  & -  & -
\end{matrix}\)$
 and $A^3 = \(\begin{matrix} abc & cd^2 & acd \\
c^2d & abc & - \\
- & bcd  & abc
\end{matrix}\),$
implying $$A^3 = \alpha A + \Det{A}$$ in this case, where $\alpha $ denotes the other coefficient in $f_A.$ But for $A = \(\begin{matrix} a & - & -\\
- & b & - \\
- & -  & c
\end{matrix}\)$ we have  $$A^3 + \alpha A + 2\(\begin{matrix} - & - & - \\
- & abc & - \\
- & -  & abc
\end{matrix}\)=  \tr(A)A^2 + \Det{A},$$ so  neither $A^3 + \alpha A $ nor $ \tr(A)A^2 + \Det{A}$ necessarily  surpasses the other.
\end{examp}

\subsection{Layered surpassing identities}\label{layPI}

%
%

Since we want to deal with general layers, we write $2a$ (instead
of~$a^\nu$) for $a+a$, but note that $s(2a) = 2s(a).$
 When working with the layered structure, we can extend the notion of PI
from Definition~\ref{PI2} by making use of the following relations that arise naturally in the theory.

\begin{defn}\label{ghostsurpL}   The \textbf{$L$-surpassing relation}
$\lmodWL$ is given by
 \begin{equation}  a \lmodWL b \quad
\text{  iff either } \quad  \begin{cases} a=b+c &  \quad
\text{with}\quad c\quad  s(b)\text{-ghost},
  \quad \\ a=b,\\   
   a\nucong b & \quad\text{with}\quad
 a \quad  s(b)\text{-ghost}.\end{cases} \end{equation} \end{defn} It follows that if
$a \lmodWL b$, then $ a+b$ is $s(b)$-ghost. When $a \ne b$, this
means $a \ge _\nu b$ and $a$ is $s(b)$-ghost.

\begin{defn} \label{PI4}  The \textbf{surpassing $(L,\nu )$-relation}
$\lmodWLnu$ is given by
 \begin{equation}  a \lmodWLnu b \qquad
\text{  iff } \qquad a \lmodWL  b \quad
\text{and}\quad
   a\nucong b .\end{equation}

 The \textbf{surpassing $L$-identity}  $f\lmodL g$
holds for $f, g\in \Fun (R^{(n)},R)$   if $f(a_1, \dots, a_n)
\lmodL g(a_1, \dots, a_n)$ for all~
 $a_1, \dots, a_n \in R$.

 The \textbf{surpassing $(L,\nu )$-identity}  $f\lmodWLnu g$
holds for $f,g \in \Fun (R^{(n)},R)$ if $f(a_1, \dots, a_n)
\lmodWLnu g(a_1, \dots, a_n)$  for all
  $a_1, \dots, a_n \in R$.
\end{defn}

\subsubsection{Layered surpassing identities of commutative layered
semirings}

Just as the Boolean algebra satisfies the PI $x^2 =x,$ we have
some surpassing identities for commutative layered \domains0.

\begin{prop} (Frobenius identity)  $(x_1+x_2)^m \lmodWLnu  x_1^m +x_2^m .$
\end{prop}
\begin{proof} This is just a restatement of
\cite[Remark~5.2]{IzhakianKnebuschRowen2011CategoriesI}.
\end{proof}

\begin{prop}\label{permprime}    $(x_1+x_2+x_3)(x_1x_3+x_2x_3+x_1x_2) \lmodWLnu  (x_1+x_2)(x_1+x_3)(x_2+x_3) .$
More generally, let $g_1 = \sum_i  x_i,$ $g_2 = \sum_{i < j }  x_i
x_j,$ $ \dots$, and $g_{m-1} = \sum_{i }\prod_{j \neq i}  x_j$.
Then

\begin{equation}\label{twofacts}
  g_1\cdots g_{m-1} \lmodWLnu  \prod_{i < j } ( x_i +  x_j) .
\end{equation}
\end{prop}
\begin{proof} This is just a restatement of
\cite[Theorem~8.51]{IzhakianRowen2007SuperTropical}.
\end{proof}

\subsection{Layered surpassing identities of
matrices}\label{matrices}$ $

%




We applied the strong transfer principle of Akian, Gaubert, and
Guterman \cite[Theorem~4.21]{AGG} to the (standard) supertropical
matrix semiring in \cite{IzhakianRowen2009MatricesII}. We would
like to make a similar argument  in the layered case, but  must
avoid the following kind of counterexamples, pointed out by Adi
Niv:

\begin{examp}\label{exm:8.1} Suppose $A = \(\begin{matrix} \xl{10}{1} & \xl{4}{2} \\
  \xl{4}{2} & \xl{0}{10}
\end{matrix}\).$
Then $A^2 = \(\begin{matrix} \xl{20}{1} & \xl{14}{2} \\
  \xl{14}{2} & \xl{8}{4}
\end{matrix}\),$
so $|A| = \xl{10}{10}$ whereas  $|A^2| = \xl{28}{8}$, which does
not $\Net$-surpass $|A|^2$ (and does not even $\Net$-surpass
$|A|$).
\end{examp}

The difficulty in the example was that some $\nu$-small entry of
$A$ has a high layer which provides $|A|$ a high layer but does
not affect the powers of $A$. There is a  version of surpassing
which is useful   in this context.

\begin{defn} An element $c \in  R$ is a
\textbf{strong} $\ell$\textbf{-ghost} (for $\ell \in L_+)$ if
$s(c) \ge 2\ell$.

The \textbf{strong $\ell$-surpassing relation} $a \SlmodWl b$
holds in an $L$-layered \domain0 $R$, if  either
\begin{equation}
\begin{cases} a=b+c &  \quad
\text{with}\quad c \quad  \text{a strong }  \ell \text{-ghost}
  \quad \\ \quad\text{or} \\  a=b.\end{cases} \end{equation}
We often take $\ell = s(b).$ In this case $b+b \SlmodWl b$ (as
well as $b+b \lmodWl b$).


 The \textbf{strong $\ell$-surpassing relation} $(a_{i,j}) \SlmodWl
(b_{i,j})$   holds for matrices $(a_{i,j})$ and $(b_{i,j})$, if
$a_{i,j} \SlmodWl b_{i,j}$ for each $i,j$. \end{defn}

We say that a matrix $A$ is $\ell$-
\textbf{layered} if each entry has layer $\ge \ell$. We are ready for our other two versions of layered identities.

\begin{defn}

 The \textbf{strong $(\ell,d)$-surpassing identity}  $f \SdltmodWl g$
holds for $f,g \in \Fun (M_n(R)^{(m)},M_n(R))$ if $f(A_1, \dots,
A_m) \SltmodWl g(A_1, \dots, A_m)$ with $\tilde \ell = \ell^d,$
for all $\ell$-layered matrices $A_1, \dots, A_m \in M_n(R)$.
\end{defn}

 In
the standard supertropical theory we take $\ell = \tilde \ell =1,$
but in the general layered theory we may need to consider other
$\ell$.
Formally set $P(x_1, \dots, x_\ell)= P^+ - P^-$ and $ Q(x_1,
\dots, x_\ell)= Q^+ - Q^-$.
 We say $Q$ is \textbf{admissible} if
the monomials of $Q^+$ and $Q^-$ are distinct, for each pair
$(i,j)$.

We then obtain the following metatheorem, along the lines of
\cite{AGG} (just as in
\cite[Theorem~2.4]{IzhakianRowen2009MatricesII}):

\begin{thm}\label{STP1} Suppose $P=Q$ is a homogeneous matrix identity of $M_n(\mathbb Z)$ of
degree $d$, with $Q$ admissible.  Then   the
 matrix \semiring0 $M_n(R)$ satisfies the strong $(\ell,d)$-surpassing
identity $$P ^+ + P^- \SdltmodWl  Q^+ + Q^-.$$
\end{thm}

 Here are some applications.

\begin{cor} $\Det{AB} \SdltmodWl \Det{A}\Det{B}$ for $L$-layered
$n\times n$ matrices $A$ and $B$, where  $d =  2n$.\end{cor}

%

Given an  $L$-layered matrix $A$ and the polynomial $$f_A :=
\Det{\la I + A} = \al_n \lm^n + \cdots + \al_1 \lm + \al_0,$$ we
define the polynomial $\widetilde {f_A}$ to be
$$ \Inu{ f_A } =  \Inu{\al}_n \lm^{n-1} + \cdots +  \Inu{ \al}_2 \lm +
 \Inu{ \al}_1,$$
where $s(\Inu{\al}_i) = \ell^{n-i}$ and  $\Inu{ \al}_i \nucong
\al_i$.
\begin{thm}\label{adjform0} $\widetilde{
f_A}(A) \SdltmodWl \adj{A}$,   where  $d=
 n-1$, for any $\ell$-layered matrix  $A$. \end{thm}
\begin{proof} This is an identity for $M_n(\mathbb Z),$ using the
usual determinant.  \end{proof}

%
%

\begin{prop}\label{adjform1} $\adj{\adj{A}} \SdltmodWl \Det{A}^{n-2}A$, where $d =
 n-1$, for any $\ell$-layered matrix  $A$.
\end{prop}

Questions for further thought:

\begin{enumerate}
   \item[Q1.] What are all the \semiring0 PIs of $M_n(R)$?

Specifically, we have the Specht-like question:

  \item[Q2.]
Are all \semiring0 PIs of  $M_n(R)$ a consequence of a given
finite set?
\end{enumerate}

\begin{examp} It is shown in \cite{IzhakianMargolisIdentity} that the semiring of $2\times 2$ matrices over
the max-plus algebra satisfies the semigroup identity
\begin{equation}\label{eq:2X2} A
B^2 A \ A B \ A B^2 A \ds =  A B^2 A \  BA  \ A B^2 A.
\end{equation}
The way  of proving this identity is essentially based on showing
that pairs of polynomials corresponding to compatible entries in
the right and the left product above define the same function.
This identification is performed  by using the machinery of Newton
polytopes, and thus is valid also for supertropical polynomials.
From the results of \cite{IzhakianMargolisIdentity}, we also
conclude that this identity is minimal.
\end{examp}

\end{document}